\documentclass[a4paper]{amsart}
\usepackage{graphicx}
\usepackage{amsmath}
\usepackage{amssymb}
\usepackage{oldgerm}
\usepackage[all]{xy}

\newcommand{\Hom}{\operatorname{Hom}\nolimits}
\renewcommand{\Im}{\operatorname{Im}\nolimits}

\newcommand{\Ker}{\operatorname{Ker}\nolimits}

\newcommand{\Pd}{\operatorname{pd}\nolimits}

\newcommand{\Ht}{\operatorname{ht}\nolimits}
\newcommand{\depth}{\operatorname{depth}\nolimits}

\newcommand{\Tor}{\operatorname{Tor}\nolimits}
\newcommand{\Ext}{\operatorname{Ext}\nolimits}
\newcommand{\Spec}{\operatorname{Spec}\nolimits}

\newcommand{\m}{\operatorname{\mathfrak{m}}\nolimits}
\newcommand{\az}{\operatorname{\mathfrak{a}}\nolimits}
\newcommand{\pr}{\operatorname{\mathfrak{p}}\nolimits}

\newcommand{\cx}{\operatorname{cx}\nolimits}

\newcommand{\cidim}{\operatorname{CI-dim}\nolimits}

\newcommand{\p}{\operatorname{p}\nolimits}
\newcommand{\q}{\operatorname{q}\nolimits}

\newtheorem{theorem}{Theorem}[section]

\newtheorem{corollary}[theorem]{Corollary}

\newtheorem{lemma}[theorem]{Lemma}
\newtheorem{proposition}[theorem]{Proposition}
\newtheorem{definition}[theorem]{Definition}
\theoremstyle{definition}
\newtheorem*{example}{Example}

\theoremstyle{definition}

\theoremstyle{remark}
\newtheorem*{remark}{Remark}
\theoremstyle{definition}
\newtheorem*{assumption}{Assumption}

\begin{document}
\title{Modules with reducible complexity}
\author{Petter Andreas Bergh}
\address{Petter Andreas Bergh \\ Institutt for matematiske fag \\
NTNU \\ N-7491 Trondheim \\ Norway} \email{bergh@math.ntnu.no}
%\date{\today}
\maketitle

\begin{abstract}
For a commutative Noetherian local ring we define and study the
class of modules having reducible complexity, a class containing all
modules of finite complete intersection dimension. Various
properties of this class of modules are given, together with results
on the vanishing of homology and cohomology.
\end{abstract}

\section{Introduction}

The complexity of a finitely generated module over a commutative
Noetherian local ring is a measure on a polynomial scale of the
growth of the ranks of the free modules in its minimal free
resolution. Over a local complete intersection every finitely
generated module has finite complexity. In fact, it follows from
\cite[Theorem 2.3]{Gulliksen} that this property characterizes a
local complete intersection, and that it is equivalent to the
complexity of the residue field being finite.

In \cite{Avramov3} a certain finiteness condition was defined under
which a module \emph{behaves} homologically like a module over a
complete intersection. Namely, a module $M$ over a commutative
Noetherian local ring $A$ has finite \emph{complete intersection
dimension} if there exist local rings $R$ and $Q$ and a diagram $A
\to R \twoheadleftarrow Q$ of local homomorphisms (called a
\emph{quasi-deformation} of $A$) such that $A \to R$ is faithfully
flat, $R \twoheadleftarrow Q$ is surjective with kernel generated by
a regular sequence, and $\Pd_Q (R \otimes_A M)$ is finite. There is
of course a reason behind the choice of terminology; it was shown in
\cite{Avramov3} that a local ring is a complete intersection if and
only if all its finitely generated modules have finite complete
intersection dimension, and that this is equivalent to the
finiteness of the complete intersection dimension of the residue
field. Moreover, it was shown that if the projective dimension of a
module is finite, then it is equal to the complete intersection
dimension of the module.

We shall study a class of modules whose complexity is ``reducible"
(defined in the next section), a class which contains all modules of
finite complete intersection dimension. In particular, we
investigate for this class of modules the vanishing of homology and
cohomology, and generalize several of the results in \cite{Araya},
\cite{Choi}, \cite{Jorgensen1} and \cite{Jorgensen2}.

\section{reducible complexity}

Throughout this paper we let ($A, \m, k$) be a commutative
Noetherian local ring, and we suppose all modules are finitely
generated. We fix a finitely generated nonzero $A$-module $M$ with
minimal free resolution
$$\bold{F}_M \colon \cdots \to F_2 \to F_1
\to F_0 \to M \to 0.$$ The rank of $F_n$, i.e.\ the integer $\dim_k
\Ext_A^n (M,k)$, is the $n$'th \emph{Betti number} of $M$, and we
denote this by $\beta_n(M)$. The \emph{complexity} of $M$, denoted
$\cx M$, is defined as
$$\cx M = \inf \{ t \in \mathbb{N}_0 \mid \exists a
\in \mathbb{R} \text{ such that } \beta_n(M) \leq an^{t-1} \text{
for } n \gg 0 \},$$ where $\mathbb{N}_0 = \mathbb{N} \cup \{ 0 \}$.
In general the complexity of a module may be infinite, whereas it is
zero if and only if the module has finite projective dimension.

We now give the main definition, which is inductive and defines the
class of modules having reducible complexity as a subcategory of the
category of (finitely generated) $A$-modules having finite
complexity. However, before stating the definition, recall that the
$n$'th \emph{syzygy} of an $A$-module $X$ with minimal free
resolution
$$\cdots \to P_2 \to P_1 \to P_0 \to X \to 0$$
is the cokernel of $P_{n+1} \to P_n$ and denoted by $\Omega_A^n (X)$
(note that $\Omega_A^0 (X)=X$), and it is unique up to isomorphism.
For an $A$-module $Y$ and a homogeneous element $\eta \in
\Ext^*_A(X,Y)$, choose a map $f_{\eta} \colon \Omega_A^{|\eta|}(X)
\to Y$ representing $\eta$, and denote by $K_{\eta}$ the pushout of
this map and the inclusion $\Omega_A^{|\eta|}(X) \hookrightarrow
P_{|\eta|-1}$. As parallel maps in a pushout diagram have isomorphic
cokernels, we obtain an exact sequence
$$0 \to Y \to K_{\eta} \to \Omega_A^{|\eta|-1}(X) \to 0.$$
Note that the middle term $K_{\eta}$ is independent (up to
isomorphism) of the map $f_{\eta}$ chosen to represent $\eta$.

\begin{definition}\label{def}
Let $\mathcal{C}_A$ denote the category of all $A$-modules having
finite complexity. The subcategory $\mathcal{C}_A^r \subseteq
\mathcal{C}_A$ of modules having \emph{reducible complexity} is
defined inductively as follows:
\begin{enumerate}
\item[(i)] Every module of finite projective dimension belongs to
$\mathcal{C}_A^r$. \item[(ii)]A module $X \in \mathcal{C}_A$ with
$\cx X >0$ belongs to $\mathcal{C}_A^r$ if there exists a
homogeneous element $\eta \in \Ext^*_A(X,X)$ of positive degree
such that $\cx K_{\eta} < \cx X, \depth K_{\eta} = \depth X$ and
$K_{\eta} \in \mathcal{C}_A^r$. We say that the element $\eta$
\emph{reduces the complexity} of $M$.
\end{enumerate}
\end{definition}

\begin{remark}\sloppy
The condition $\depth K_{\eta} = \depth X$ is not very strong;
suppose for example that $\depth \Omega_A^i(X) \leq \depth A$ for
all $i$ (this always happens when $A$ is Cohen-Macaulay). Then we
must have $\depth X \leq \depth \Omega_A^{|\eta|-1} (X)$, implying
the equality $\depth K_{\eta} = \depth X$.
\end{remark}

Note the trivial fact that if every $A$-module has reducible
complexity, then $A$ must be a complete intersection since then by
definition every module has finite complexity. The following
result shows that the converse is also true, in fact every module
of finite complete intersection dimension has reducible
complexity. Moreover, if $A$ is Cohen-Macaulay and $M$ has
reducible complexity, then so does any syzygy of $M$.

\begin{proposition}\label{prop2.2}
\begin{enumerate}
\item[(i)] Every module of finite complete intersection dimension
has reducible complexity. \item[(ii)] If $A$ is Cohen-Macaulay and
$M$ has reducible complexity, then so does the kernel of any
surjective map $F \twoheadrightarrow M$ when $F$ is free. In
particular, any syzygy of $M$ has reducible complexity.
\item[(iii)] if $B$ is a local ring and $A \to B$ a faithfully flat local
homomorphism, then if $M$ has reducible complexity, so does the
$B$-module $B \otimes_A M$.
\end{enumerate}
\end{proposition}

\begin{proof}
(i) If $M$ has finite complete intersection dimension, then from
\cite[Proposition 5.2]{Avramov3} it follows that the complexity of
$M$ is finite. We argue by induction on $\cx M $ that $M$ has
reducible complexity, the case $\cx M=0$ following from the
definition. Suppose therefore that the complexity of $M$ is nonzero.
By \cite[Proposition 7.2]{Avramov3} there exists an eventually
surjective chain map of degree $-n$ (where $n>0$) on the minimal
free resolution of $M$. This chain map corresponds to an element
$\eta$ of $\Ext_A^n (M,M)$, giving the exact commutative diagram
$$\xymatrix{
0 \ar[r] & \Omega_A^n (M) \ar[r] \ar[d]^{f_{\eta}} & F_{n-1}
\ar[r] \ar[d] & \Omega_A^{n-1} (M) \ar[r] \ar@{=}[d] & 0 \\
0 \ar[r] & M \ar[r] & K_{\eta} \ar[r] & \Omega_A^{n-1} (M) \ar[r]
& 0 }$$ of $A$-modules.

Now consider the exact sequence involving $K_{\eta}$. Since $M$
and $\Omega_A^{n-1}(M)$ both have finite complete intersection
dimension, then so does $K_{\eta}$. The exact sequence gives rise
to a long exact sequence
$$\cdots \to \Ext_A^i (K_{\eta},k) \to \Ext_A^i (M,k)
\xrightarrow{\partial_{\eta}} \Ext_A^{i+n} (M,k) \to \Ext_A^{i+1}
(K_{\eta},k) \to \cdots,$$ where $\Ext_A^i ( \Omega_A^{n-1} (M) ,k)$
has been identified with $\Ext_A^{i+n-1} (M,k)$. It follows from
\cite[Theorem III.9.1]{MacLane} that the connecting homomorphism
$\partial_{\eta} \colon \Ext_A^i (M,k) \to \Ext_A^{i+n} (M,k)$ is
then scalar multiplication by $(-1)^i \eta$ (we think of
$\Ext_A^*(M,k)$ as a graded right module over the graded ring
$\Ext_A^*(M,M)$), and reversing the arguments in the proof of
\cite[Proposition 2.2]{Bergh} shows that this is an injective map
for $i \gg 0$. Consequently the equality $\beta_{i+1} (K_{\eta}) =
\beta_{i+n} (M) - \beta_i (M)$ holds for large values of $i$. By
\cite[Theorem 5.3]{Avramov3} the complexities $\cx M$ and $\cx
K_{\eta}$ equal the orders of the poles at $t=1$ of the Poincar\'e
series $\sum \beta_n (M)t^n$ and $\sum \beta_n (K_{\eta})t^n$,
respectively, thus $\cx K_{\eta} = \cx M-1$.

It remains only to show that $\depth K_{\eta} = \depth M$, but
this is easy; from \cite[Theorem 1.4]{Avramov3} we have $0 \leq
\cidim \Omega_A^i(M) = \depth A - \depth \Omega_A^i(M)$ for each
$i \geq 0$, and this implies the inequality $\depth \Omega_A^i(M)
\leq \depth \Omega_A^{i+1} (M)$. In particular we have $\depth M
\leq \depth \Omega_A^{n-1} (M)$, and therefore $\depth K_{\eta}$
must equal $\depth M$.

(ii) Let $L$ denote the kernel of the surjective map $F
\twoheadrightarrow M$. Again we argue by induction on $\cx M$. If
the projective dimension of $M$ is finite, then so is the
projective dimension of $L$, and we are done. Suppose therefore
$\cx M$ is nonzero, and let $\eta \in \Ext_A^*(M,M)$ be an element
reducing the complexity of $M$. By the Horseshoe Lemma we have an
exact commutative diagram
$$\xymatrix{
& 0 \ar[d] & 0 \ar[d] & 0 \ar[d] \\
0 \ar[r] & L \ar[d] \ar[r] & K \oplus Q \ar[d] \ar[r] &
\Omega_A^{|\eta|-1} (L) \ar[d] \ar[r] & 0 \\
0 \ar[r] & F \ar[d] \ar[r] & F' \oplus Q \ar[d] \ar[r] & F''
\oplus Q \ar[d] \ar[r] & 0 \\
0 \ar[r] & M \ar[d] \ar[r] & K_{\eta} \ar[d] \ar[r] &
\Omega_A^{|\eta|-1}(M) \ar[d] \ar[r] & 0 \\
& 0 & 0 & 0 }$$ where $F'$ and $F''$ are free modules, and $Q$ is
a free module such that $\Omega_A^{|\eta|}(M) \oplus Q \simeq
\Omega_A^{|\eta|-1} (L)$. If $M$ and $K_{\eta}$ are maximal
Cohen-Macaulay, then so are $L$ and $K \oplus Q$, and if not then
$\depth L = \depth M+1$ and $\depth (K \oplus Q) = \depth
K_{\eta}+1$. In either case we see that $\depth L$ equals $\depth
(K \oplus Q)$. Moreover, we have $\cx (K \oplus Q) = \cx K_{\eta}
< \cx M = \cx L$, and so by induction we are done.

Note that the upper horizontal short exact sequence in the above
diagram corresponds to an element $\theta$ in $\Ext_A^{|\eta|}(L,L)$
reducing the complexity of $L$. A map $f_{\theta} \colon
\Omega_A^{|\eta|} (L) \to L$ representing $\theta$ is obtained by
lifting a map $f_{\eta}$ representing $\eta$ along the minimal free
resolution of $M$, as in the diagram
$$\xymatrix{
0 \ar[r] & \Omega_A^{|\eta|+1} (M) \ar@{.>}[d]^{\Omega_A(f_{\eta})}
\ar[r] & F_{|\eta|} \ar@{.>}[d] \ar[r] & \Omega_A^{|\eta|} (M)
\ar[d]^{f_{\eta}} \ar[r] & 0 \\
0 \ar[r] & L \ar[r] & F \ar[r] & M \ar[r] & 0 }$$ In this way we
obtain a map $\Omega_A(f_{\eta}) \colon \Omega_A^{|\eta|+1}(M) \to
L$, and adding to $\Omega_A^{|\eta|+1}(M)$ a free module $Q'$ such
that $\Omega_A^{|\eta|+1}(M) \oplus Q' \simeq \Omega_A^{|\eta|}
(L)$, we obtain a map $\Omega_A^{|\eta|} (L) \to L$ representing
$\theta$.

(iii) If $X$ is any $A$-module with a minimal free resolution
$\bold{F}_X$, then the complex $B \otimes_A \bold{F}_X$ is a minimal
free $B$-resolution of $B \otimes_A X$, giving the equality $\cx_A X
= \cx_B (B \otimes_A X)$. Moreover, if $Y$ is another $A$-module
then from \cite[Theorem 23.3]{Matsumura} we have $\depth_A X -
\depth_A Y = \depth _B (B \otimes_A X) - \depth_B (B \otimes_A Y)$.
Hence if $\eta \in \Ext_A^*(M,M)$ reduces the complexity of $M$, the
element $B \otimes_A \eta \in \Ext_B^*(B \otimes_A M,B \otimes_A M)$
reduces the complexity of $B \otimes_A M$.
\end{proof}

Thus the class of modules having finite complete intersection
dimension is contained in the class of modules having reducible
complexity. However, the following example shows that the inclusion
is strict in general; there are a lot of modules having reducible
complexity but whose complete intersection dimension is infinite.

\begin{example}
Suppose $M$ is periodic of period $p \geq 3$ (i.e.\ $\Omega_A^p (M)$
is isomorphic to $M$), and that $\depth M \leq \depth \Omega_A^{p-1}
(M)$. Then we have an exact sequence
$$0 \to M \to F_{p-1} \to \Omega_A^{p-1} (M) \to 0,$$
and we have $0= \cx F_{p-1} = \cx M-1$ and $\depth F_{p-1} =
\depth M$. Therefore $M$ has reducible complexity, and it cannot
be of finite complete intersection dimension since then by
\cite[Theorem 7.3]{Avramov3} the period would have been two.

An example of such a module was given in \cite[Section3]{Gasharov};
let $(A, \m, k)$ be the commutative local finite dimensional
$k$-algebra $k[x_1, x_2, x_3, x_4]/ \az$, where $\az$ is the ideal
generated by the quadratic forms
$$x_1^2, \hspace{.5cm} x_2^2, \hspace{.5cm} x_3^2, \hspace{.5cm}
x_4^2, \hspace{.5cm} x_3x_4, \hspace{.5cm} x_1x_4+x_2x_4,
\hspace{.5cm} \alpha x_1x_3+x_2x_3$$ for a nonzero element $\alpha
\in k$. The complex
$$\cdots \to A^2 \xrightarrow{d_2} A^2
\xrightarrow{d_1} A^2 \xrightarrow{d_0} M \to 0,$$ where the maps
are given by the matrices
$$d_n = \left( \begin{array}{cc}
                 x_1 & \alpha^n x_3 + x_4 \\
                 0 & x_2
               \end{array} \right), $$
is a minimal free resolution of the module $M := \Im d_0$, and so if
$\alpha$ has finite order $p$ we see that $M$ is periodic of period
$p$.

It is worth mentioning that there do exist examples of modules
over Gorenstein rings whose complete intersection dimension is not
finite but which have reducible complexity (see \cite[Proposition
3.1]{Gasharov} for an example similar to that above).
\end{example}

\sloppy Now let $X$ and $Y$ be arbitrary $A$-modules, and $\theta_1
\in \Ext_A^* (X,X)$ and $\theta_2 \in \Ext_A^* (X,Y)$ two
homogeneous elements. The following lemma, motivated by \cite[Lemma
7.2]{Erdmann}, links $K_{\theta_1}$ and $K_{\theta_2}$ to
$K_{\theta_2 \theta_1}$, and will be a key ingredient in several of
the forthcoming results.

\begin{lemma}\label{lem2.3}
If $\theta_1 \in \Ext_A^* (X,X)$ and $\theta_2 \in \Ext_A^* (X,Y)$
are two homogeneous elements, then there exists an exact sequence
$$0 \to \Omega_A^{|\theta_2|}(K_{\theta_1}) \to K_{\theta_2
\theta_1}\oplus F \to K_{\theta_2} \to 0$$ of $A$-modules, where $F$
is free.
\end{lemma}

\begin{proof}
Denote $|\theta_i|$ by $n_i$ for $i=1,2$. The element $\theta_1$
gives rise to an exact commutative diagram
$$\xymatrix{
0 \ar[r] & \Omega_A^{n_1}(X) \ar[d]^{f_{\theta_1}} \ar[r]^i &
Q_{n_1-1} \ar[d]^h \ar[r]^{\pi} &
\Omega_A^{n_1-1}(X) \ar@{=}[d] \ar[r] & 0 \\
0 \ar[r] & X \ar[r]^j & K_{\theta_1} \ar[r] & \Omega_A^{n_1-1}(X)
\ar[r] & 0 }$$ where $Q_n$ denotes the $n$'th module in the minimal
free resolution of $X$. Letting $Q \xrightarrow{g} X$ be a
surjection, where $Q$ is free, we can modify the diagram and obtain
$$\xymatrix{
& 0 \ar[d] & 0 \ar[d] \\
& \Ker (h,jg) \ar[d] \ar@{=}[r] & \Ker (h,jg) \ar[d] \\
 0 \ar[r] & \Omega_A^{n_1}(X) \oplus Q
\ar@{->>}[d]^{(f_{\theta_1},g)} \ar[r]^{\left ( \begin{smallmatrix}i
& 0
\\ 0 & 1
\end{smallmatrix} \right )} & Q_{n_1-1} \oplus Q
\ar@{->>}[d]^{(h,jg)} \ar[r]^{(\pi,0)} &
\Omega_A^{n_1-1}(X) \ar@{=}[d] \ar[r] & 0 \\
0 \ar[r] & X \ar[r]^j \ar[d] & K_{\theta_1} \ar[r] \ar[d] &
\Omega_A^{n_1-1}(X) \ar[r] & 0 \\
& 0 & 0 }$$ Since $\Ker (h,jg)$ is isomorphic to $\Omega_A^1
(K_{\theta_1}) \oplus Q'$ for some free module $Q'$, the left
vertical exact sequence yields an exact sequence
$$0 \to \Omega_A^1 (K_{\theta_1}) \oplus Q' \to
\Omega_A^{n_1}(X) \oplus Q \xrightarrow{(f_{\theta_1},g)} X \to 0$$
on which we can apply the Horseshoe Lemma and obtain an exact
sequence
$$\mu \colon 0 \to \Omega_A^{n_2}(K_{\theta_1}) \to
\Omega_A^{n_1+n_2-1}(X) \oplus F
\xrightarrow{(\Omega_A^{n_2-1}(f_{\theta_1}),s)} \Omega_A^{n_2-1}(X)
\to 0,$$ where $F$ is free and $F \xrightarrow{s}
\Omega_A^{n_2-1}(X)$ is a map.

The definition of cohomological products and the pushout properties
of $K_{\theta_2 \theta_1}$ give a commutative diagram
$$\xymatrix{
0 \ar[r] & Y \ar[r] \ar@{=}[d] & K_{\theta_2 \theta_1}
\ar[r]^<<<<<<w \ar[d]^t & \Omega_A^{n_1+n_2-1}(X) \ar[r]
\ar[d]^{\Omega_A^{n_2-1}(f_{\theta_1})} & 0 \\
0 \ar[r] & Y \ar[r] & K_{\theta_2} \ar[r]^v & \Omega_A^{n_2-1}(X)
\ar[r] & 0 }$$ with exact rows. Adding $F$ to $K_{\theta_2
\theta_1}$ and $\Omega_A^{n_1+n_2-1}(X)$ in the right-most square,
we obtain the exact commutative diagram
$$\xymatrix{
&& 0 \ar[d] & 0 \ar[d] \\
&& \Omega_A^{n_2}(K_{\theta_1}) \ar@{=}[r] \ar[d] &
\Omega_A^{n_2}(K_{\theta_1}) \ar[d] \\
0 \ar[r] & Y \ar[r] \ar@{=}[d] & K_{\theta_2 \theta_1} \oplus F
\ar[r]^>>>>>{\left ( \begin{smallmatrix}w & 0 \\ 0 & 1
\end{smallmatrix} \right )} \ar[d]^{(t,r)} &
\Omega_A^{n_1+n_2-1}(X) \oplus F \ar[r]
\ar[d]^{(\Omega_A^{n_2-1}(f_{\theta_1}),s)} & 0 \\
0 \ar[r] & Y \ar[r] & K_{\theta_2} \ar[d] \ar[r]^>>>>>>>v &
\Omega_A^{n_2-1}(X) \ar[d] \ar[r] & 0 \\
&& 0 & 0 }$$ in which the right vertical exact sequence is $\mu$ and
$F \xrightarrow{r} K_{\theta_2}$ is a map with the property that
$s=vr$. The left vertical exact sequence is of the form we are
seeking.
\end{proof}

We end this section with two results on modules over complete
intersections. Recall that a maximal Cohen-Macaulay (or ``MCM"
from now on) \emph{approximation} of an $A$-module $X$ is an exact
sequence
$$0 \to Y_X \to C_X \to X \to 0$$
where $C_X$ is MCM and $Y_X$ has finite injective dimension, and
that a \emph{hull of finite injective dimension} of $X$ is an exact
sequence
$$0 \to X \to Y^X \to C^X \to 0$$
where $C^X$ is MCM and $Y^X$ has finite injective dimension. These
notions were introduced in \cite{Auslander2}, where it was shown
that every finitely generated module over a commutative Noetherian
ring admitting a dualizing module has a MCM approximation and a hull
of finite injective dimension. The following result provides a
simple proof of the complete intersection case, using a technique
similar to the proof of the main result in \cite{Bakke} and the fact
that over a complete intersection every module has reducible
complexity.

\begin{proposition}\label{prop2.4}
Suppose $A$ is a complete intersection.
\begin{enumerate}
\item[(i)] If $\eta \in \Ext_A^{|\eta|} (M,M)$ reduces the
complexity of $M$, then so does $\eta^t$ for $t \geq 1$.
\item[(ii)] Every $A$-module has a MCM approximation and a hull of
finite injective dimension.
\end{enumerate}
\end{proposition}

\begin{proof}
(i) Using Lemma \ref{lem2.3} it is easily proved by induction on $t$
that $\cx K_{\eta^t} \leq \cx K_{\eta}$.

(ii) Fix an exact sequence
$$0 \to Y \to C \to M \to 0$$
where $C$ is MCM (the minimal free cover of $M$, for example). If
the complexity of $Y$ is nonzero, let $\eta \in \Ext_A^{|\eta|}
(Y,Y)$ be an element reducing it, and choose an integer $t \geq 1$
with the property that $\Omega_A^{t |\eta| -1}(Y)$ is MCM. The
element $\eta^t$ is given by the exact sequence
$$0 \to Y \to K_{\eta^t} \to \Omega_A^{t |\eta| -1}(Y) \to 0,$$
and by (i) it also reduces the complexity of $Y$. From the pushout
diagram
$$\xymatrix{
& 0 \ar[d] & 0 \ar[d] \\
0 \ar[r] & Y \ar[d] \ar[r] & K_{\eta^t} \ar[d] \ar[r] & \Omega_A^{t
|\eta| -1}(Y) \ar@{=}[d] \ar[r] & 0 \\
0 \ar[r] & C \ar[d] \ar[r] & C' \ar[d] \ar[r] & \Omega_A^{t |\eta|
-1}(Y) \ar[r] & 0 \\
& M \ar[d] \ar@{=}[r] & M \ar[d] \\
& 0 & 0 }$$ we obtain the exact sequence
$$0 \to Y' \to C' \to M \to 0$$
(with $Y' = K_{\eta^t}$), where $C'$ is MCM and $\cx Y' < \cx Y$.
Repeating the process we eventually obtain a MCM approximation of
$M$, since over a Gorenstein ring a module has finite injective
dimension precisely when it has finite projective dimension.

As for a hull of finite injective dimension, fix an exact sequence
$$0 \to M \to Y \to C \to 0$$
where $C$ is MCM (obtained for example from a suitable power of an
element in $\Ext_A^* (M,M)$ reducing the complexity of $M$). If the
complexity of $Y$ is nonzero, choose as above an element $\eta \in
\Ext_A^{|\eta|} (Y,Y)$ reducing it, and let $t \geq 1$ be an integer
such that $\Omega_A^{t |\eta| -1}(Y)$ is MCM. From the pushout
diagram
$$\xymatrix{
& 0 \ar[d] & 0 \ar[d] \\
& M \ar[d] \ar@{=}[r] & M \ar[d] \\
0 \ar[r] & Y \ar[d] \ar[r] & K_{\eta^t} \ar[d] \ar[r] & \Omega_A^{t
|\eta| -1}(Y) \ar@{=}[d] \ar[r] & 0 \\
0 \ar[r] & C \ar[d] \ar[r] & C' \ar[d] \ar[r] & \Omega_A^{t |\eta|
-1}(Y) \ar[r] & 0 \\
& 0 & 0 }$$ we obtain the exact sequence
$$0 \to M \to Y' \to C' \to 0$$
(with $Y' = K_{\eta^t}$), where $C'$ is MCM and $\cx Y' < \cx Y$.
Repeating the process we eventually obtain a hull of finite
injective dimension of $M$.
\end{proof}

\section{Vanishing results}

This section investigates the vanishing of cohomology and homology
for a module having reducible complexity, and \emph{the following
is assumed throughout}:

\begin{assumption}
The module $M$ has reducible complexity, and $N$ is a nonzero
$A$-module. If $\cx M>0$, then there exist $A$-modules $K_1,
\dots, K_c$ and a set of cohomological elements $\{ \eta_i \in
\Ext_A^{|\eta_i|}(K_{i-1},K_{i-1}) \}_{i=1}^c$ given by exact
sequences
$$0 \to K_{i-1} \to K_i \to \Omega_A^{|\eta_i|-1}(K_{i-1}) \to 0$$
for $i=1, \dots, c$ (where $K_0=M$), satisfying $\depth K_i =
\depth M, \cx K_i < \cx K_{i-1}$ and $\cx K_c =0$ (such elements
$\eta_i$ must exist by Definition \ref{def}).
\end{assumption}

For an $A$-module $N$, we define $\q^A (M,N)$ and $\p^A (M,N)$ by
\begin{eqnarray*}
\q^A (M,N) &=& \sup \{ n \mid \Tor_n^A (M,N) \neq 0 \}, \\
\p^A (M,N) &=& \sup \{ n \mid \Ext^n_A (M,N) \neq 0 \}.
\end{eqnarray*}
The definition of modules having reducible complexity suggests
that when proving results about $\q^A (M,N)$ and $\p^A (M,N)$, we
use induction on the complexity of $M$.

The first result and its corollary (which considers a conjecture
of Auslander and Reiten) consider the vanishing of cohomology, and
generalize \cite[Theorem 4.2 and Theorem 4.3]{Araya}.

\begin{theorem}\label{thm3.1}
The following are equivalent.
\begin{enumerate}
\item[(i)] There exists an integer $t > \depth A - \depth M$ such
that $\Ext_A^{t+i} (M,N)=0$ for $0 \leq i \leq |\eta_1| + \cdots +
|\eta_c|-c$. \item[(ii)] $\p^A (M,N) < \infty.$ \item[(iii)] $\p^A
(M,N) = \depth A - \depth M.$
\end{enumerate}
\end{theorem}

\begin{proof}
We only need to show the implication (i) $\Rightarrow$ (iii), and
we do this by induction on $\cx M$. If the projective dimension of
$M$ is finite, then by the Auslander-Buchsbaum formula it is equal
to $\depth A - \depth M$. Since $N$ is finitely generated, we have
$N \neq \m N$ by Nakayama's Lemma, hence there exists a nonzero
element $x \in N \setminus \m N$. The map $A \to N$ defined by $1
\mapsto x$ then gives rise to a nonzero element of $\Ext_A^{\Pd M}
(M,N)$, and therefore $\p^A (M,N) = \depth A - \depth M$.

Now suppose the complexity of $M$ is nonzero, and consider the
exact sequence
\begin{equation*}\label{ES2}
0 \to M \to K_1 \to \Omega_A^{|\eta_1|-1} (M) \to 0.
\tag{$\dagger$}
\end{equation*}
The vanishing interval for $\Ext_A^i (M,N)$ implies that
$\Ext_A^{t+i} (K_1,N)=0$ for $0 \leq i \leq |\eta_2| + \cdots +
|\eta_c|-(c-1)$, and so by induction $\p^A (K_1,N)= \depth A -
\depth K_1$. Since we have equalities $\p^A(M,N) = \p^A (K_1,N)$
and $\depth M = \depth K_1$, we are done.
\end{proof}

\begin{corollary}\label{cor3.2}
$\p^A (M,M)= \Pd M$.
\end{corollary}

\begin{proof}
Suppose $\p^A (M,M)< \infty$. If the projective dimension of $M$ is
not finite, i.e.\ if $\cx M >0$, then consider the exact sequence
(\ref{ES2}) representing $\eta_1$ (from the proof of Theorem
\ref{thm3.1}), where $K_1 = K_{\eta_1}$. Since $\eta_1$ is nilpotent
there is an integer $t$ such that $\eta_1^t =0$, and therefore $\cx
K_{\eta_1^t} = \cx M$. But using Lemma \ref{lem2.3} we see that $\cx
K_{\eta_1^i} \leq \cx K_{\eta_1}$ for all $i \geq 1$, and since $\cx
K_{\eta_1} < \cx M$ we have reached a contradiction. Therefore the
projective dimension of $M$ is finite and equal to $\depth A -
\depth M$ by the Auslander-Buchsbaum formula, and from Theorem
\ref{thm3.1} we see that $\p^A (M,M)= \Pd M$.
\end{proof}

The next result is a homology version of Theorem \ref{thm3.1}, and
it is closely related to \cite[Theorem 2.1]{Jorgensen1}.

\begin{theorem}\label{thm3.3}
The following are equivalent.
\begin{enumerate}
\item[(i)] There exists an integer $t > \depth A - \depth M$ such
that $\Tor^A_{t+i} (M,N)=0$ for $0 \leq i \leq |\eta_1| + \cdots +
|\eta_c|-c$. \item[(ii)] $\q^A (M,N) < \infty.$ \item[(iii)]
$\depth A - \depth M - \depth N \leq \q^A (M,N) \leq \depth A -
\depth M.$
\end{enumerate}
\end{theorem}

\begin{proof}
We only need to show the implication (i) $\Rightarrow$ (iii), and
we do this by induction on $\cx M$. The case $\Pd M < \infty$
follows from the Auslander-Buchsbaum formula and \cite[Remark
8]{Choi}, so suppose therefore $\cx M >0$, and consider the exact
sequence (\ref{ES2}) from the proof of Theorem \ref{thm3.1}. Since
$\Tor_{t+i}^A (K_1,N)=0$ for $0 \leq i \leq |\eta_2| + \cdots +
|\eta_c|-(c-1)$, we get by induction that the inequalities hold
for $K_1$ and $N$. But as in the previous proof we have $\q^A(M,N)
= \q^A (K_1,N)$ and $\depth M = \depth K_1$, hence the
inequalities hold for $M$ and $N$.
\end{proof}

The following result contains half of \cite[Theorem 3]{Choi} (and a
version of the first half of \cite[Theorem 2.5]{Araya}) and the main
result from \cite{Jorgensen2} for Cohen-Macaulay rings, which says
that the integer $\q^A (M,N)$ can be computed locally. It also
establishes the \emph{depth formula} provided $N$ is maximal
Cohen-Macaulay.

\begin{theorem}\label{thm3.4}
\begin{enumerate}
\item[(i)] If $\q^A (M,N)$ is finite and $\depth \Tor_{\q^A
(M,N)}^A (M,N) =0$, then
$$\q^A (M,N) = \depth A - \depth M - \depth N.$$
\item[(ii)] If $A$ is Cohen-Macaulay and $\q^A (M,N)$ is finite,
then the equality
$$\q^A (M,N) = \sup \{ \Ht \pr - \depth M_{\pr} - \depth
N_{\pr} \mid \pr \in \Spec A \}$$ holds. \item[(iii)] If $\q^A
(M,N)=0$ and $N$ is maximal Cohen-Macaulay, then the depth formula
holds for $M$ and $N$, i.e.\
$$\depth M + \depth N = \depth A + \depth (M \otimes N).$$
\end{enumerate}
\end{theorem}

\begin{proof}
(i) We argue by induction on $\cx M$, the case $\Pd M < \infty$
following from \cite[Theorem 1.2]{Auslander} and the
Auslander-Buchsbaum formula. Suppose therefore that the complexity
of $M$ is nonzero, and consider the exact sequence (\ref{ES2})
from the proof of Theorem \ref{thm3.1}. Since $\Tor_{\q^A (M,N)}^A
(M,N)$ is a submodule of $\Tor_{\q^A (M,N)}^A (K_1,N)$, the latter
is also of depth zero, hence by induction and the equalities $\q^A
(K_1,N)= \q^A (M,N), \depth K_1 = \depth M$ we are done.

(ii) Suppose $\q^A (M,N)$ is finite, and let $\pr \subseteq A$ be a
prime ideal. If $M$ has finite projective dimension, so has
$M_{\pr}$, and from Theorem \ref{thm3.3} we get $\q^{A_{\pr}}
(M_{\pr},N_{\pr}) \geq \Ht \pr - \depth M_{\pr} - \depth N_{\pr}$.
If $\cx M >0$, consider the exact sequences
$$0 \to K_{i-1} \to K_i \to \Omega_A^{|\eta_i|-1}(K_{i-1}) \to 0$$
for $i=1, \dots, c$ (where $K_0=M$), satisfying $\depth K_i =
\depth M, \cx K_i < \cx K_{i-1}$ and $\cx K_c =0$. Localizing at
$\pr$, we see that $\depth (K_i)_{\pr} = \depth M_{\pr}$ (as in
the remark following Definition \ref{def}), and that $\q^{A_{\pr}}
((K_i)_{\pr},N_{\pr}) = \q^{A_{\pr}} ((K_{i-1})_{\pr},N_{\pr})$.
As $K_c$ has finite projective dimension we get
\begin{eqnarray*}
\q^{A_{\pr}} (M_{\pr},N_{\pr}) & = & \q^{A_{\pr}}
((K_c)_{\pr},N_{\pr}) \\
& \geq & \Ht \pr - \depth (K_c)_{\pr} - \depth N_{\pr} \\
& = & \Ht \pr - \depth M_{\pr} - \depth N_{\pr},
\end{eqnarray*}
hence since $\q^A (M,N) \geq \q^{A_{\pr}} (M_{\pr},N_{\pr})$ the
inequality
$$\q^A (M,N) \geq \sup \{ \Ht \pr - \depth M_{\pr} - \depth
N_{\pr} \mid \pr \in \Spec A \}$$ holds.

For the reverse inequality, choose any associated prime $\pr$ of
$\Tor_{\q^A (M,N)}^A(M,N)$. Then $\q^A (M,N) = \q^{A_{\pr}}
(M_{\pr},N_{\pr})$ and $\depth
\Tor_{\q^{A_{\pr}}(M_{\pr},N_{\pr})}^{A_{\pr}}(M_{\pr},N_{\pr})
=0$, and a small adjustment of the proof of (i) above gives
$$\q^{A_{\pr}} (M_{\pr},N_{\pr}) = \Ht \pr - \depth M_{\pr} -
\depth N_{\pr}.$$

(iii) Again we argue by induction on $\cx M$, where the case $\Pd
M < \infty$ follows from \cite[Theorem 1.2]{Auslander}. Suppose
$\cx M>0$, and consider the exact sequence (\ref{ES2}). We have
$\q^A (K_1,N)=0$, hence by induction the depth formula holds for
$K_1$ and $N$. Since $\depth K_1 = \depth M$, we only have to
prove that the equality $\depth (M \otimes N) = \depth (K_1
\otimes N)$ holds. For each $i \geq 0$ we have an exact sequence
$$0 \to \Omega_A^{i+1}(M) \otimes N \to F_i \otimes N \to
\Omega_A^i(M) \otimes N \to 0,$$ and as $N$ is maximal
Cohen-Macaulay we must have that $\depth \left ( \Omega_A^i(M)
\otimes N \right )$ is at most $\depth \left ( \Omega_A^{i+1}(M)
\otimes N \right )$. In particular the inequality $\depth (M
\otimes N) \leq \depth \left ( \Omega_A^{|\eta_1|-1}(M) \otimes N
\right )$ holds, and therefore when tensoring the sequence
(\ref{ES2}) with $N$ we see that $\depth (M \otimes N) = \depth
(K_1 \otimes N)$.
\end{proof}

\sloppy The next result deals with symmetry in the vanishing of
$\Ext$. It was shown in \cite{Avramov2} that if $X$ and $Y$ are
modules over a complete intersection $A$, then $\Ext_A^i(X,Y)=0$ for
$i \gg 0$ if and only if $\Ext_A^i(Y,X)=0$ for $i \gg 0$. This was
generalized in \cite{HunekeJorgensen} to a class of local Gorenstein
rings named ``AB rings", a class properly containing the class of
complete intersections. Another generalization appeared in
\cite{PJorgensen}, where techniques from the theory of derived
categories were used to show that symmetry in the vanishing of
$\Ext$ holds for modules of finite complete intersection dimension
over local Gorenstein rings.

\begin{theorem}\label{thm3.5}
If $A$ is Gorenstein then the implication
$$\p^A (N,M) < \infty \Rightarrow \p^A (M,N) < \infty$$
holds. In particular, symmetry in the vanishing of $\Ext$ holds for
modules with reducible complexity over a local Gorenstein ring.
\end{theorem}

\begin{proof}
Define the integer $w$ (depending on $M$) by
$$w= \depth A - \depth M + |\eta_1| + \cdots
+ |\eta_c|-c.$$ If for some integer $i \geq 1$ we have
$\Tor^A_i(M,N) = \cdots = \Tor^A_{i+w}(M,N)=0$, then
$\Tor^A_i(M,N)=0$ for all $i \geq 1$ by Theorem \ref{thm3.3}. The
result now follows from \cite[Theorem 1.7 and Proposition
2.2]{PJorgensen}.
\end{proof}

The final result deals with the vanishing of homology for two
modules when $A$ is a complete intersection. Namely, in this
situation the homology modules are given as the homology modules
of two modules of finite projective dimension, due to the fact
that every module over a complete intersection has reducible
complexity.

\begin{theorem}\label{thm3.6}
Suppose $A$ is a complete intersection, and let $X$ and $Y$ be
$A$-modules such that $\Tor_i^A(X,Y)=0$ for $i\gg 0$. Then there
exist $A$-modules $X'$ and $Y'$, both of finite projective
dimension, such that $\depth X = \depth X', \depth Y = \depth Y'$
and $\Tor_i^A(X,Y) \simeq \Tor_i^A(X',Y')$ for $i>0$.
\end{theorem}

\begin{proof}
If the complexity of one of $X$ and $Y$, say $X$, is nonzero, choose
a homogeneous element $\eta \in \Ext_A^*(X,X)$ reducing the
complexity. By Proposition \ref{prop2.4}(i) any power of $\eta$ also
reduces the complexity of $X$, so choose an integer $t$ such that
$\Tor_i^A(\Omega_A^{t |\eta|-1}(X),Y)=0$ for $i>0$. The element
$\eta^t$ is given by the short exact sequence
$$0 \to X \to K_{\eta^t} \to \Omega_A^{t |\eta|-1}(X) \to 0,$$
therefore by the choice of $t$ we see that $\Tor_i^A(X,Y)$ and
$\Tor_i^A(K_{\eta^t},Y)$ are isomorphic for $i>0$. Since $A$ is
Cohen-Macaulay we automatically have $\depth X = \depth K_{\eta^t}$,
and repeating this process we eventually obtain what we want.
\end{proof}

This result has consequences for the study of the rigidity of $\Tor$
over Noetherian local rings. This study was initiated by M.\
Auslander in his 1961 paper \cite{Auslander}, in which he proved his
famous rigidity theorem; if $X$ and $Y$ are modules over an
unramified regular local ring $R$, and $\Tor^R_n(X,Y)=0$ for some $n
\geq 1$, then $\Tor^R_i(X,Y)=0$ for all $i \geq n$ (recall that a
regular local ring $(S, \m_S)$ is said to be \emph{ramified} if it
is of characteristic zero while its residue class field has
characteristic $p >0$ and $p$ is an element of $\m_S^2$). In 1966
S.\ Lichtenbaum extended Auslander's rigidity theorem to all regular
local rings (see \cite{Lichtenbaum}), and subsequently Peskine and
Szpiro conjectured in \cite{PeskineSzpiro} that the theorem holds
for all Noetherian local rings provided one of the modules in
question has finite projective dimension. A counterexample to the
conjecture was provided by Heitmann in \cite{Heitmann}, where a
Cohen-Macaulay ring $R$ together with $R$-modules $X$ and $Y$ were
given, for which $\Pd X =2$ and $\Tor^R_1(X,Y)=0$, while
$\Tor^R_2(X,Y) \neq 0$.

However, whether the rigidity of $\Tor$ holds for Noetherian local
rings provided \emph{both} modules involved have finite projective
dimension is unknown. If this holds over complete intersections,
then Theorem \ref{thm3.6} shows that the conjecture of Peskine and
Szpiro also holds for such rings (i.e.\ rigidity of $\Tor$ holds
provided one of the modules involved has finite projective
dimension). In fact, the theorem shows that if rigidity holds over a
complete intersection $R$ provided both modules have finite
projective dimension, then rigidity holds over $R$ for all modules
$X$ and $Y$ satisfying $\Tor^R_i(X,Y)=0$ for $i \gg 0$.

\section{A generalization}

In this final section we discuss a situation which slightly
generalizes the concept of reducible complexity. Instead of letting
$M$ have reducible complexity as in Definition \ref{def}, we make
the following assumption:

\begin{assumption}
The complexity of $M$ is finite, and if it is nonzero then there
exist local rings $\{ R_i \}_{i=1}^c$ such that for each $i \in \{
1, \dots, c \}$ there is a faithfully flat local homomorphism
$R_{i-1} \to R_i$ (where $R_0=A$), an $R_i$-module $K_i$, an
integer $n_i$ and an exact sequence
$$0 \to R_i \otimes_{R_{i-1}} K_{i-1} \to K_i \to
\Omega_{R_i}^{n_i} (R_i \otimes_{R_{i-1}} K_{i-1}) \to 0$$ (where
$K_0 =M$) satisfying $\depth_{R_i} K_i = \depth_{R_i} (R_i
\otimes_{R_{i-1}} K_{i-1}), \cx_{R_i} K_i < \cx_{R_{i-1}} K_{i-1}$
and $\Pd_{R_c} K_c < \infty$.
\end{assumption}

Of course, if $M$ has reducible complexity then by choosing each
$R_i$ to be $A$ we see that the assumption is satisfied. Now let
$S \to T$ be any faithfully flat local homomorphism, and $X$ and
$Y$ any (finitely generated) $S$-modules. If $\bold{F}_X$ is a
minimal $S$-free resolution of $X$, then the complex $T \otimes_S
\bold{F}_X$ is a minimal $T$-free resolution of $T \otimes_S X$,
and by \cite[Proposition (2.5.8)]{Grothendieck} we have natural
isomorphisms
\begin{eqnarray*}
\Hom_T (T \otimes_S \bold{F}_X, T \otimes_S Y) & \simeq & T
\otimes_S \Hom_S ( \bold{F}_X, Y), \\
(T \otimes_S \bold{F}_X) \otimes_T (T \otimes_S Y) & \simeq & T
\otimes_S ( \bold{F}_X \otimes_S Y).
\end{eqnarray*}
Therefore we have isomorphisms
\begin{eqnarray*}
\Ext_T^i (T \otimes_S X,T \otimes_S Y) & \simeq & T \otimes_S
\Ext_S^i (X,Y), \\
\Tor^T_i (T \otimes_S X,T \otimes_S Y) & \simeq & T \otimes_S
\Tor^S_i (X,Y),
\end{eqnarray*}
and as $T$ is faithfully $S$-flat we then get
\begin{eqnarray*}
\Ext_T^i (T \otimes_S X,T \otimes_S Y)=0 & \Leftrightarrow &
\Ext_S^i (X,Y) =0, \\
\Tor^T_i (T \otimes_S X,T \otimes_S Y)=0 & \Leftrightarrow &
\Tor^S_i (X,Y) =0.
\end{eqnarray*}
We then get the equalities
\begin{eqnarray*}
\cx_S X &=& \cx_T (T \otimes_S X) \\
\p^S (X,Y) &=& \p^T (T \otimes_S X, T \otimes_S Y) \\
\q^S (X,Y) &=& \q^T (T \otimes_S X, T \otimes_S Y) \\
\depth_S X - \depth_S Y &=& \depth _T (T \otimes_S X) - \depth_T
(T \otimes_S Y),
\end{eqnarray*}
where the one involving depth follows from \cite[Theorem
23.3]{Matsumura}.

Using the above facts it is easy to see that both Theorem
\ref{thm3.1} and Corollary \ref{cor3.2} remain true in this new
situation, as does Theorem \ref{thm3.3} if we drop the left
inequality in (iii).

Suppose now that $M$ has finite complete intersection dimension.
Then \cite[Proposition 7.2]{Avramov3} and an argument similar to the
proof of Proposition \ref{prop2.2}(i) show that $M$ satisfies this
new assumption \emph{and that $n_i=1$ for each $1 \leq i \leq c$}.
Consequently, the vanishing intervals in Theorem \ref{thm3.1}(i) and
Theorem \ref{thm3.3}(i) are of length $\cx_A M+1$, as in
\cite[Theorem 4.7]{Avramov2} and \cite[Theorem 2.1]{Jorgensen1}.
Moreover, we obtain \cite[Theorem 4.2]{Avramov2}, which says that
$M$ is of finite projective dimension if and only if $\Ext_A^{2n}
(M,M)=0$ for some $n \geq 1$. To see this, note that when $\cx_A M
>0$ the extension
$$0 \to R_1 \otimes_A M \to K_1 \to \Omega_{R_1}^1
(R_1 \otimes_A M) \to 0$$ corresponds to an element $\theta \in
\Ext_{R_1}^2(R_1 \otimes_A M,R_1 \otimes_A M)$. If $\Ext_A^{2n}
(M,M)=0$ for some $n \geq 1$, then $\Ext_{R_1}^{2n}(R_1 \otimes_A
M,R_1 \otimes_A M)$ also vanishes, hence $\theta^{2n}=0$. As in the
proof of Corollary \ref{cor3.2} we obtain the contradiction
\begin{eqnarray*}
\cx_A M & = & \cx_{R_1} (R_1 \otimes_A M) \\
& = & \cx_{R_1} K_{\theta^{2n}} \\
& \leq & \cx_{R_1} K_1 \\
& < & \cx_A M,
\end{eqnarray*}
showing that we cannot have $\Ext_A^{2n} (M,M)=0$ for some $n \geq
1$ when $M$ is of positive complexity.

\section*{Acknowledgements}

I would like to thank Dave Jorgensen and my supervisor {\O}yvind
Solberg for valuable suggestions and comments on this paper.


\begin{thebibliography}{EHSST}
\bibitem[Aus]{Auslander}M.\ Auslander, \emph{Modules over unramified
regular local rings}, Illinois J.\ Math.\ 5 (1961), 631-647.
\bibitem[AuB]{Auslander2}M.\ Auslander, R.-O.\ Buchweitz,
\emph{The homological theory of maximal Cohen-Macaulay
approximations}, M\'{e}m.\ Soc.\ Math.\ France 38 (1989), 5-37.
\bibitem[AvB]{Avramov2}L.\ Avramov, R.-O.\ Buchweitz,
\emph{Support varieties and cohomology over complete
intersection}, Invent.\ Math.\ 142 (2000), 285-318.
\bibitem[AGP]{Avramov3}L.\ Avramov, V.\ Gasharov, I.\ Peeva,
\emph{Complete intersection dimension}, Publ.\ Math.\ I.H.E.S.\ 86
(1997), 67-114.
\bibitem[ArY]{Araya}T.\ Araya, Y.\ Yoshino, \emph{Remarks on a
depth formula, a grade inequality and a conjecture of Auslander},
Comm.\ Algebra 26 (1998), 3793-3806.
\bibitem[Bak]{Bakke}\O.\ Bakke, \emph{The existence of short exact
sequences with some of the terms in given subcategories}, in
\emph{Algebras and modules II (Geiranger 1996)}, CMS Conf.\ Proc.\
24 (1998), 39-45.
\bibitem[Ber]{Bergh}P.A.\ Bergh, \emph{Complexity and periodicity},
Coll.\ Math. 104 (2006), no.\ 2, 169-191.
\bibitem[ChI]{Choi}S.\ Choi, S.\ Iyengar, \emph{On a depth formula
for modules over local rings}, Comm.\ Algebra 29 (2001), 3135-3143.
\bibitem[EGA]{Grothendieck}A.\ Grothendieck, \emph{\'El\'ements de
g\'eom\'etrie alg\'ebrique}, chapitre IV, seconde partie,
I.H.\'E.S.\ Publ.\ Math.\ 24 (1965).
\bibitem[EHSST]{Erdmann}K.\ Erdmann, M.\ Holloway, N.\ Snashall,
\O.\ Solberg, R.\ Taillefer, \emph{Support varieties for
selfinjective algebras}, $K$-theory 33 (2004), 67-87.
\bibitem[GaP]{Gasharov}V.\ Gasharov, I.\ Peeva, \emph{Boundedness
versus periodicity over commutative local rings}, Trans.\ Amer.\
Math.\ Soc.\ 320 (1990), 569-580.
\bibitem[Gul]{Gulliksen}T.\ Gulliksen, \emph{On the deviations of
a local ring}, Math.\ Scand.\ 47 (1980), 5-20.
\bibitem[Hei]{Heitmann}R.\ Heitmann, \emph{A counterexample to the
rigidity conjecture for rings}, Bull.\ Amer.\ Math.\ Soc.\ (N.S.) 29
(1993), no.\ 1, 94-97.
\bibitem[HuJ]{HunekeJorgensen}C.\ Huneke, D.\ Jorgensen,
\emph{Symmetry in the vanishing of $\Ext$ over Gorenstein rings},
Math.\ Scand.\ 93 (2003), no.\ 2, 161-184.
\bibitem[Jo1]{Jorgensen1}D.\ Jorgensen, \emph{Complexity and $\Tor$
on a complete intersection}, J.\ Algebra 211 (1999), 578-598.
\bibitem[Jo2]{Jorgensen2}D.\ Jorgensen, \emph{A generalization of
the Auslander-Buchsbaum formula}, J.\ Pure Appl.\ Algebra 144
(1999), 145-155.
\bibitem[J{\o}r]{PJorgensen}P.\ J{\o}rgensen, \emph{Symmetry
theorems for $\Ext$ vanishing}, J.\ Algebra 301 (2006), 224-239.
\bibitem[Lic]{Lichtenbaum}S.\ Lichtenbaum, \emph{On the vanishing of
$\Tor$ in regular local rings}, Illinois J.\ Math.\ 10 (1966),
220-226.
\bibitem[Mac]{MacLane}S.\ Mac Lane, \emph{Homology}, Classics in
Mathematics, Springer-Verlag, 1995.
\bibitem[Mat]{Matsumura}H.\ Matsumura, \emph{Commutative ring
theory}, Cambridge University Press, 2000.
\bibitem[PeS]{PeskineSzpiro}C.\ Peskine, L.\ Szpiro, \emph{Dimension
projective finie et cohomologie locale. Applications \`a la
d\'emonstration de conjectures de M.\ Auslander, H.\ Bass et A.\
Grothendieck}, I.H.\'E.S.\ Publ.\ Math.\ 42 (1973), 47-119.
\end{thebibliography}
\end{document}